\documentclass[11pt]{article}
\usepackage{amsmath,amsfonts,amssymb,latexsym,epsfig}
\usepackage{color,epsf}
\usepackage{graphicx,stix}

\usepackage{subfig}
\usepackage{alltt}
\usepackage{url}
\usepackage{a4}

\newcommand{\N}{\mathbb{N}}

\newtheorem{lemma}{Lemma}

\newtheorem{proposition}{Proposition}

\newtheorem{remark}{Remark}
\newtheorem{definition}{Definition}

\newenvironment{proof}{\paragraph{Proof:}}{$\Box$}

\newcommand{\e}{\ensuremath{\mathrm{e}}}

\newcommand{\psis}{{\mathcal S}}

\makeindex
\begin{document}
\title{On symmetric-conjugate composition methods in the numerical integration of differential 
equations
}

\author{S. Blanes\thanks{Instituto de Matem\'atica Multidisciplinar, Universitat Polit\`ecnica de Val\`encia, 46022-Valencia, Spain.
Email: \texttt{serblaza@imm.upv.es}}
 \and
 F. Casas\thanks{Departament de Matem\`atiques and IMAC, Universitat Jaume I, E-12071 Castell\'on, Spain. Email: \texttt{Fernando.Casas@uji.es} }
 \and
 P. Chartier\thanks{Universit\'e de Rennes, INRIA, CNRS, IRMAR, F-35000 Rennes, France. Email: \texttt{Philippe.Chartier@inria.fr}} 
\and 
A. Escorihuela-Tom\`as\thanks{Departament de Matem\`atiques and IMAC, 
Universitat Jaume I, E-12071 Castell\'on, Spain. Email: \texttt{alescori@uji.es}}
 }


%
\maketitle

\begin{abstract}

We analyze composition methods with complex coefficients exhibiting the so-called ``symmetry-conjugate'' pattern in their distribution. In particular,
we study
their behavior with respect to preservation of qualitative properties when projected on the real axis and we compare them with the
usual left-right palindromic compositions. New schemes within this family up to order 8 are proposed and their efficiency is tested on several
examples. Our analysis shows that higher-order schemes are more efficient even when time step sizes are relatively large.

\vspace*{1cm}


\end{abstract}\bigskip

\noindent AMS numbers: 65L05, 65P10, 37M15


\section{Introduction}
\label{sec.1}

We are concerned in this work with compositions of  a time-symmetric 2nd-order integrator, denoted by $\mathcal{S}_h^{[2]}$. 
To be more specific, given the initial value problem
\begin{equation}  \label{ivp}
  x' = f(x), \qquad x(t_0) = x_0 \in \mathbb{R}^d
\end{equation}
with solution $x(t) = \varphi_t(x_0)$, method $\mathcal{S}_h^{[2]}$ verifies that $\mathcal{S}_h^{[2]}(x_0) = \varphi_h(x_0) + \mathcal{O}(h^3)$  for a time step $h$ and moreover $\mathcal{S}_h^{[2]} \circ \mathcal{S}_{-h}^{[2]} = \mathrm{id}$,
the identity map, for any $h$.  Then, the $s$-stage composition methods we are considering here are of the form
\begin{equation}  \label{compo.1}
  \psi_h^{[r]} = \mathcal{S}_{\alpha_s h}^{[2]} \circ \mathcal{S}_{\alpha_{s-1} h}^{[2]} \circ \cdots \circ \mathcal{S}_{\alpha_2 h}^{[2]} \circ \mathcal{S}_{\alpha_1 h}^{[2]},
\end{equation}
where the coefficients $\alpha_j$ are certain numbers chosen in such a way that the order of approximation of $\psi_h^{[r]}$ is $r \ge 2$. 

Methods 
(\ref{compo.1})  constitute a very efficient class of numerical integrators for (\ref{ivp}), especially when $f$ has some geometric
properties that is advantageous to preserve under discretization. In fact, composition methods preserve any group properties 
shared by the basic scheme $\mathcal{S}_h^{[2]}$:
symplecticity, phase space volume, first integrals, symmetries, etc., and therefore they provide a general and flexible class of geometric numerical
integrators \cite{hairer06gni}. 

If in addition the sequence of coefficients in (\ref{compo.1}) is \emph{left-right palindromic}, i.e., $\alpha_{s+1-j} = \alpha_j$, $j=1,2,\ldots$,
then $\psi_h^{[r]}$ is also time-symmetric, i.e., it verifies for small $h$
\begin{equation}   \label{t-s}
  \psi_h^{[r]} \circ \, \psi_{-h}^{[r]} = \mathrm{id},
\end{equation}
and are of even order, $r=2n$ \cite{hairer06gni}. 
Methods of this class are called symmetric compositions of symmetric schemes  \cite{mclachlan95otn} and constitute an efficient way to
construct high-order approximations, due to the reduction in the
number of order conditions to be satisfied.

Nevertheless, the fact that composition methods of order greater than 2 require some negative coefficients $\alpha_j$ typically imposes
 severe stability restrictions on the time step,
especially when dealing with semidiscretized PDEs \cite{blanes05otn}. To try to remedy this situation, \emph{complex} coefficients with positive real part have been considered in the literature for this class of problems \cite{blanes13oho,castella09smw,hansen09hos}. In fact, methods with
complex coefficients have also been used even for problems when the presence of negative fractional time steps is not problematic
\cite{bandrauk06cis,chambers03siw}.

If composition methods with complex coefficients are applied to a real vector field $f$ in (\ref{ivp}), then the approximation
 $x_1$ at the end of the first time step $t_1=t_0 + h$ will be
of course complex, whereas the exact solution is real. A relevant issue is then how to proceed with the computation of the trajectory. Two possibilities exist: either one pursues the determination
of the solution for $t > t_1$ with the previously computed value of $x_1 \in \mathbb{C}$ and project on the real axis only when output is desired (after, say,
$N$ integration steps) or one just discards the imaginary part of $x_1$ and initiates the
next step only with $\Re(x_1)$. In both cases, however, the favourable properties the composition inherits from the basic scheme 
$\mathcal{S}_h^{[2]}$ (such as symplecticity) are most often lost. 
Previous (heuristic) analyses show that, generally speaking, the later approach provides a better description of the problem 
\cite{blanes13oho,blanes10smw,chambers03siw}.

One purpose of this work is to provide a rigorous justification of this observation and determine up to what degree symplecticity, say, is still preserved
when using complex coefficients. We show, in particular, that a $2n$-th order left-right palindromic composition 
with complex coefficients, when projected on the real axis after each step, still preserves the time-symmetry and other relevant
geometric properties up to order $4n+1$. Moreover, we also show that it is possible to preserve the time-symmetry up to a higher order
by considering another family
of compositions, namely methods of the form (\ref{compo.1}) with
the special symmetry
\begin{equation}   \label{ss.1}
    \alpha_{s+1-j} = \bar{\alpha}_j, \qquad j=1,2, \ldots,
\end{equation}
where $\bar{\alpha}_j$ denotes the complex conjugate of $\alpha_j$. For obvious reasons, we call the resulting scheme  
\begin{equation}  \label{compo.2}
  \psi_h^{[r]} = \mathcal{S}_{\bar{\alpha}_1 h}^{[2]} \circ \mathcal{S}_{\bar{\alpha}_{2} h}^{[2]} \circ \cdots \circ 
  \mathcal{S}_{\alpha_2 h}^{[2]} \circ \mathcal{S}_{\alpha_1 h}^{[2]},
\end{equation}
a \emph{symmetric-conjugate} composition.
 The simplest method within this family is of course
\begin{equation}  \label{s3a}
   \psi_h^{[3]} = \mathcal{S}_{\alpha h}^{[2]} \circ    \mathcal{S}_{\bar{\alpha} h}^{[2]}.
\end{equation}
If
\[
    \alpha = \frac{1}{2} \pm i \frac{\sqrt{3}}{6},
\]
then $\psi_h^{[3]}$  is of order 3, but if one considers instead only its real part,
\begin{equation}   \label{ss34}
   \Re(  \psi_h^{[3]} ) = \frac{1}{2} \left(  \psi_h^{[3]} +  \overline{\psi}_h^{[3]}    \right) = \hat{R}^{[4]}_h,
\end{equation}
or equivalently, if one projects $\psi_h^{[3]}$ at each time step on the real axis,
then the resulting scheme $\hat{R}^{[4]}_h$ is an integration method of order 4. This fact has been previously recognized  by several authors 
\cite{bandrauk06cis,chambers03siw}.
Although $\hat{R}^{[4]}_h$ is no longer time-symmetric, it nevertheless verifies 
\[
   \hat{R}^{[4]}_{-h} \circ \hat{R}^{[4]}_h = \mathrm{id} + \mathcal{O}(h^8)
\]      
when the vector field $f$ in (\ref{ivp}) is real \cite{casas21cop}. Moreover,  if $f$ is a (real) Hamiltonian vector field 
and $\mathcal{S}_h$ is a 2nd-order symplectic integrator, then $\hat{R}^{[4]}_h$ is also symplectic with an error $\mathcal{O}(h^8)$. 

Motivated by this feature of scheme $\hat{R}^{[4]}_h$ and the excellent preservation properties of methods (\ref{compo.2})  reported in
particular in \cite{blanes10smw},
we shall analyze in detail this class of integrators. In doing so, we will pay special attention to their preservation properties, and 
eventually we will propose new schemes requiring less number of stages for achieving a given order than left-right palindromic compositions
when projected on the real axis after each integration step.

\section{Compositions of a second-order symmetric scheme}
\label{sec.2}

\subsection{Integrators and series of operators}

If $\varphi_{h}$ is the exact flow of the equation (\ref{ivp}), then for each  infinitely differentiable map $g$,
 the function $g(\varphi_{h}(x))$ admits an expansion of the form \cite{arnold89mmo,sanz-serna94nhp}
\[
g(\varphi_{h}(x)) = \exp(h F)[g](x) = g(x) + \sum_{k\geq 1}
\frac{h^k}{k!} F^k[g](x), 
\]
where $F$ is the Lie derivative associated with $f$,
\begin{equation}  \label{eq.1.1b}
  F = \sum_{i\ge1} \, f_i(x) \, \frac{\partial }{\partial x_i}.
\end{equation}
Analogously, for the class of integrators $\psi_{h}$ we are considering,  one can associate a series of linear operators so that
\[
    g(\psi_{h}(x)) = \exp(Y(h))[g](x), \quad \mbox{ with }  \quad Y(h) = \sum_{j \ge 1} h^j Y_j
\]
for all functions $g$ \cite{blanes08sac}. Here $Y_{j}$ are operators depending on the particular method considered.
The integrator $\psi_{h}$ is of order $r$ if
\[ 
   Y_1 = F \qquad \mbox{ and } \qquad Y_j = 0 \;\; \mbox{ for } \;\; 2 \le j \le r.
\]   
For the adjoint integrator, defined as $\psi_{h}^* := \psi_{-h}^{-1}$, one clearly has
\[
   g(\psi_{h}^*(x)) = \exp \big(-Y(-h) \big)[g](x).
\]   
Notice that $\psi_{h}$ is time-symmetric, i.e., it verifies (\ref{t-s}), if and only if $\psi^*_h = \psi_h$,  and therefore $Y(h)$ only contains odd powers of $h$.
In particular, time-symmetric methods are of even order.

According with these comments, the time-symmetric 2nd-order scheme $\psis_h^{[2]}$ can be associated with the series
\begin{equation}  \label{gc.1}
   \Phi^{[2]}(h) = \exp(h F + h^3 Y_3 + h^5 Y_5 + \cdots + h^{2k+1} Y_{2k+1} + \cdots).
\end{equation}
Then, the series  of operators associated with the integrator (\ref{compo.1}) can be determined by applying the Baker--Campbell--Hausdorff
formula, thus resulting in
\begin{equation}   \label{vh.1}
   \Psi^{[r]}(h) = \exp( V(h) ),
\end{equation}   
where $V(h)$ is formally given by
\[
   V(h) = h w_{1} F + h^3 w_{3,1} Y_3 + h^4 w_{4,1} [F, Y_3] + h^5 \big(w_{5,1} Y_5 + w_{5,2}  [F,[F,Y_3]] \big) + \mathcal{O}(h^6).
\]
Here $[F,Y_3]$ stands for the Lie bracket of the operators $F$ and $Y_3$, etc. and 
\begin{equation}  \label{orcon1}
\begin{aligned}
   & w_1 = \sum_{j=1}^s \alpha_j, \qquad\quad   w_{3,1} = \sum_{j=1}^s \alpha_j^3,  \qquad\quad w_{5,1} = \sum_{j=1}^s \alpha_j^5, \\
   &  w_{4,1} = \frac{1}{2} \sum_{j=1}^{s-1} \left( \alpha_j^3 \left( \sum_{k=j+1}^s \alpha_k \right) -  \alpha_j \left( \sum_{k=j+1}^s \alpha_k^3 \right) \right) \\
   & w_{5,2} = \frac{1}{12} \sum_{j=1}^s \alpha_j^3 \left( \left( \sum_{k = 1}^{j-1} \alpha_{k} \right)^2 + \left( \sum_{k = j+1}^{s} \alpha_{k} \right)^2 -
      4 \sum_{k=1}^{j-1} \alpha_k \sum_{\ell=j+1}^s \alpha_{\ell} \right) \\
   &  \qquad\quad  -\frac{1}{12} \sum_{j=1}^s \alpha_j^4 \left( \sum_{k=1}^{j-1} \alpha_k + \sum_{k=j+1}^s \alpha_k \right).   
 \end{aligned}   
\end{equation}
(In the expression of $w_{5,2}$ above the sum is zero when the upper index is smaller than the lower index). In general, $V(h)$ is an element of the free Lie algebra $\mathcal{L}$ generated by $\{F, Y_3, Y_5, \ldots \}$ \cite{munthe-kaas99cia}, 
i.e., $V(h)$ is a linear combination
of $F, Y_3, Y_5, \ldots$, and all their nested Lie brackets,
\begin{equation}  \label{vh.2}
  V(h) =  h w_1 F + \sum_{n \ge 3} h^{n} \sum_{k=1}^{c(n)}  w_{n,k} E_{n,k}.
\end{equation}
Here $w_{n,k}$ are polynomials in the coefficients of the method, $E_{2n+1,1} = Y_{2n+1}$ and 
$E_{n,k}$, $k > 1$, are independent nested Lie brackets of $\{F, Y_3, Y_5, \ldots \}$ forming a basis of the homogeneous component $\mathcal{L}_n$
of  $\mathcal{L}$, with dimension $c(n)$ \cite{mclachlan02sm}. 
Thus, in particular, $\mathcal{L}_5$ has dimension $c(5) = 2$, and a basis is given by 
$\{E_{5,1} =Y_5,  E_{5,2} =[F,[F,Y_3]] \}$.

Method (\ref{compo.1}) is of order $r$ if $w_1 = 1$ and the polynomials $w_{n,k}$ vanish whenever $1 < n  \le r$, and $k=1,\ldots, c(n)$. 
The number of the resulting equations (the order
conditions) $N^{[r]}$ agrees of course with the sum of the dimensions $c(n)$, i.e.,
\[
 N^{[r]} = \sum_{n=1}^r c(n)
\] 
and is collected in Table \ref{tabla1} (second row) for the first values of $r$. A composition without any special symmetry would involve then at
least $s = N^{[r]}$ stages so as to have enough parameters to solve the equations.

\begin{table} 
\begin{center}
  \begin{tabular}{|c|cccccccc|} \hline
   \textit{Order} $r$ & 1 & 2 & 3 & 4 & 5 & 6 & 7  &  8  \\ \hline\hline
 $N^{[r]}$ (General) & 1 & 0 & 2 & 3 & 5 & 7 &  11 & 16   \\ \hline
 $N_P^{[r=2n]}$ (Palindromic) &  & 1 &  & 2 &  & 4 &   & 8   \\ \hline
 $N^{[r]} - c(2n)$ (Sym-Conjugate) &  & 1 &  & 2 &  & 5 &   & 11(9)   \\ \hline 
\end{tabular}
\end{center}
  \caption{\small{Total number of order conditions to achieve order $r$ for the method resulting from projecting after each step (i) the general composition (\ref{compo.1}) of time-symmetric 2nd-order methods (second row), a left-right palindromic composition (third row) and 
  a symmetric-conjugate composition (fourth row). It turns out that by solving only 9 order conditions one can achieve order 8 with symmetric-conjugate
  compositions.}}
  \label{tabla1}
\end{table}

\subsection{Left-right palindromic compositions}

Before establishing general results about preservation of properties of composition methods with complex coefficients after projection on the
real axis, it is worth to introduce the following definitions, as in \cite{casas21cop}:
\begin{definition}  \label{def1}
Let $\psi_h$ be a smooth and consistent integrator. Then
\begin{enumerate}
\item $\psi_h$ is said to be  {\em pseudo-symmetric} of {\em pseudo-symmetry order} $q$ if for all sufficiently small $h$, it is true that
\begin{align}
\label{eq:pseudosym0}
\psi_h^* = \psi_{h} + {\cal O}(h^{q+1}),
\end{align}
where the constant in the ${\cal O}$-term depends on bounds of derivatives of $\psi_h$. 
\item $\psi_h$  is  said to be {\em pseudo-symplectic} of {\em pseudo-symplecticity order} $p$ if for all sufficiently small $h$, the following relation holds true
when it is applied to a Hamiltonian system:
\begin{align}
\label{eq:pseudosymplec}
(\psi_h')^T \, J \, \psi_{h}' = J + {\cal O}(h^{p+1}),
\end{align}
where $J$ denotes the canonical symplectic matrix and the constant in the ${\cal O}$-term depends on bounds of derivatives of $\psi_h$.
\end{enumerate}
\end{definition}
\begin{remark}
A symmetric method is pseudo-symmetric of any order $q \in \N$, whereas a method of order $r$ is pseudo-symmetric of order $q \geq r$. 
A similar statement holds for symplectic methods. 
\end{remark}

We first proceed with left-right palindromic compositions. According to the considerations in the previous section, the 
series of operators  associated with such a method of order $2n$
is $\Psi^{[2n]}(h) = \exp( V(h) )$, with
\begin{equation}  \label{vhlr}
  V(h) =  h w_1 F + \sum_{j \ge n} h^{2j+1} \sum_{k=1}^{c(2j+1)} w_{2j+1,k} E_{2j+1,k} 
 \end{equation} 
and $w_{2j+1,k}$ have in general real and imaginary parts when $\alpha_j \in \mathbb{C}$. Then one has the following

\begin{proposition}   \label{prop.3}
Given $\psis_{h}^{[2]}$ a time-symmetric 2nd-order method, consider the left-right palindromic composition
\begin{equation}   \label{compo.2b}
   \psis_{h}^{[r]} =  \mathcal{S}_{\alpha_1 h}^{[2]} \circ \mathcal{S}_{\alpha_{2} h}^{[2]} \circ \cdots \circ \mathcal{S}_{\alpha_2 h}^{[2]} \circ \mathcal{S}_{\alpha_1 h}^{[2]}
\end{equation}
 of order $r=2n$, $n =2,3,\ldots$, when the coefficients $\alpha_j$ are complex numbers satisfying $2 (\alpha_1 + \alpha_2 + \cdots) = 1$. Then 
the method obtained by taking its real part,
\[
       \phi_h^{[2n]} \equiv   \frac{1}{2} (  \psis_h^{[2n]} +  \bar{\psis}_h^{[2n]} ),
\]
is of the same order $r=2n$  and pseudo-symmetric of order $q=4n+1$
  when the vector field $f$ in  (\ref{ivp}) is real.
  If in addition $f$ is a (real) Hamiltonian vector field and $\psis_{h}^{[2]}$ is a symplectic integrator, then $\phi_{h}^{[2n]}$ is pseudo-symplectic
of order $p=4n+1$.
\end{proposition}

\begin{proof}
In this and the remaining proofs we apply a similar approach as in \cite{blanes99eos} for determining the pseudo-symplectic character of methods obtained by
polynomial extrapolation. An important ingredient is the symmetric BCH formula \cite{blanes16aci}: given $X$ and $Y$ two non-commuting
operators, then
\[
   \exp( \frac{1}{2} X) \, \exp(Y) \, \exp(\frac{1}{2} X) = \exp(Z),
\]
where $Z = \sum_{n \ge 0} Z_{2n+1}$ and $Z_{2n+1}$, $n \ge 1$, is a linear combination of nested brackets involving $2n+1$ operators
$X$ and $Y$. The first terms read
\[
  Z_1 = X + Y, \qquad\qquad Z_3 = -\frac{1}{24} [X,[X,Y]] - \frac{1}{12} [Y,[X,Y]].
\]     
To begin with, we write the expression (\ref{vhlr}) associated with (\ref{compo.2b}) as 
\[
   V(h) = h F + h^{2n+1} V_{2n+1} + h^{2n+3} V_{2n+3}   + \cdots
\]
where $V_{2n+j}$, $j=1,3,\ldots$ are \emph{complex} operators in the free Lie algebra
 generated by $\{ F, Y_{3}, Y_{5}, \ldots \}$. In consequence, the series corresponding to $\phi_{h}^{[2n]}$ is
\[
    \Phi^{[2n]}(h) = \frac{1}{2}  \exp((V(h)) +  \frac{1}{2}  \exp(\overline{V}(h)),
\]
which can be written in fact as 
\begin{equation} \label{vv.1}
    \Phi^{[2n]}(h) =  \exp \left( \frac{h}{2} F \right) \left(  \frac{1}{2}  \exp((W(h))   + \frac{1}{2} \exp(\overline{W}(h))
      \right) \exp \left( \frac{h}{2} F \right), 
\end{equation}
where $W(h)$ is determined by applying the symmetric BCH formula to $\exp(W(h)) = \exp(-h F/2) \, \exp(V(h)) \, \exp(-h F/2)$, thus
leading to
\[
     W(h)  =  h^{2n+1} V_{2n+1} + h^{2n+3} \left( V_{2n+3} + \frac{1}{24} [F,[F,V_{2n+1}]] \right) + \mathcal{O}(h^{2n+5}).
\]
Now the idea is to write $\Phi^{[2n]}(h)$ in (\ref{vv.1}) as $\e^{h F/2} \, \e^{(W + \overline{W})/2} \, \e^{h F/2} + \mathcal{O}(h^q)$,
for some $q$. Therefore, we have to analyze $\frac{1}{2} (\e^W + \e^{\overline{W}}) - \e^{(W + \overline{W})/2}$. To this end,
first we note that 
\[
  W(h) + \overline{W}(h) = 2 h^{2n+1} \Re(V_{2n+1}) + 2 h^{2n+3} \left( \Re(V_{2n+3}) + \frac{1}{24} [F,[F,\Re(V_{2n+1})]] \right) +  \mathcal{O}(h^{2n+5}),
\]
i.e., only contains odd powers of $h$ and
\[
 \frac{1}{8} \big( W(h) + \overline{W}(h) \big)^2  =  \frac{1}{2} h^{4n+2} \, \Re(V_{2n+1})^2 +    \mathcal{O}(h^{4n+4}), 
\]
whereas
\[
 \frac{1}{4} \big( W(h)^2 + \overline{W}(h)^2 \big) =    \frac{1}{2} h^{4n+2} \left( \Re(V_{2n+1})^2 - \Im(V_{2n+1})^2 \right) +    \mathcal{O}(h^{4n+4}). 
\]
In consequence, 
\[
\begin{aligned}
  & \frac{1}{2} \left( \e^{W(h)}   +  \e^{\overline{W}(h)} \right) - \e^{ \frac{1}{2} \big( W(h) + \overline{W}(h) \big)} = 
   \frac{1}{4} \big( W(h)^2 + \overline{W}(h)^2 \big) - \frac{1}{8}  \big( W(h) + \overline{W}(h) \big)^2 +  \mathcal{O}(h^{4n+4}) \\
   & \quad =  - \frac{1}{2} h^{4n+2} \Im(V_{2n+1})^2  +  \mathcal{O}(h^{4n+4})
\end{aligned}   
\]
so that
\[
   \Phi^{[2n]}(h)  =  \exp \left( \frac{h}{2} F \right)  \exp \left( \frac{1}{2} \big( W(h) + \overline{W}(h) \big) \right) \exp \left( \frac{h}{2} F \right)
   +  \mathcal{O}(h^{4n+2}) 
\]
whence the following statements follow at once:
\begin{itemize}
  \item   Method (\ref{compo.2b}) is of order $2n$, since $ \Phi^{[2n]}(h)  = \exp( h F) +  \mathcal{O}(h^{2n+1})$.
  \item Since $Z = (W(h) + \overline{W}(h))/2$ only contains odd powers of $h$, 
  then $\e^{\frac{h}{2} F} \e^{Z} \e^{\frac{h}{2} F}$ is a time-symmetric composition and $\phi_{h}^{[2n]}$
   is pseudo-symmetric of order $4n+1$.
  \item Let us suppose that scheme (\ref{compo.2b}) is applied to a Hamiltonian system and that $\psis_{h}^{[2]}$ is symplectic. 
  Since $Z$ is an operator in the free Lie algebra
  generated by $\{ F, Y_{3}, Y_{5}, \ldots \}$, clearly the composition  $\e^{\frac{h}{2} F} \e^{Z} \e^{\frac{h}{2} F}$ is symplectic.
  As a matter
  of fact, this can be extended to any geometric property the differential equation (\ref{ivp}) has: volume-preserving, unitary, etc., as long as the basic scheme
  $\psis_{h}^{[2]}$ preserves this property.
\end{itemize}
\end{proof}

\

As an example, let us consider the well known 4th-order palindromic scheme
used in the triple jump procedure \cite{yoshida90coh},
\begin{equation}  \label{13b}
   \psis_h^{[4]} =  \mathcal{S}_{\alpha_1 h}^{[2]} \circ  \, \mathcal{S}_{\alpha_2 h}^{[2]} \circ \, \mathcal{S}_{\alpha_1 h}^{[2]},
\end{equation}
with
\[
   \alpha_1 = \frac{1}{2 - 2^{1/3} \e^{2 i k \pi/3}}, \qquad \alpha_2 = 1- 2 \alpha_1, \qquad k =1,2.
\]
(Note that with $k=0$ one gets the usual real solution). Then $ \phi_h^{[4]} = \Re( \psis_h^{[4]})$ is still a method of order 4, but pseudo-symmetric and
pseudo-symplectic of order 9.

As is well known, the number of order conditions required by left-right palindromic compositions to achieve order $r=2n$ is \cite{mclachlan02sm}
\[
N_P^{[2n]}=\sum_{j=1}^nc(2j-1).
\]
In consequence, a palindromic composition requires at least $2N_P^{[2n]}-1$ stages so as to have the same number of (complex) coefficients 
as order conditions. The values of  $N_P^{[2n]}$ to achieve orders $2n = 2, 4, 6, 8$ are
collected in the third row of Table \ref{tabla1}.

\subsection{Symmetric-conjugate compositions}

In contrast with left-right palindromic compositions, even and odd order methods  of the form (\ref{compo.2}) 
exist, but their behavior with respect to structure preservation is different when they are projected on the real axis at each step. The reason
lies in the special structure of the associated series of differential operators. More specifically, we have the following
 
 \begin{lemma} \label{lem.1}
  Let $\psis_{h}^{[2]}$ be a time-symmetric 2nd-order method for (\ref{ivp}), with $f$ real, and consider the composition method
  \[
  \psi_h^{[r]} = \mathcal{S}_{\alpha_s h}^{[2]} \circ \mathcal{S}_{\alpha_{s-1} h}^{[2]} \circ \cdots \circ \mathcal{S}_{\alpha_2 h}^{[2]} \circ \mathcal{S}_{\alpha_1 h}^{[2]},
 \]
 verifying
\[
\alpha_{s+1-j} = \bar{\alpha}_j, \qquad j=1,2, \ldots.
\]
Then $ \psi_h^{[r]}$ has an associated series of operators  $\Psi^{[r]}(h) = \exp( V(h) )$, with
\begin{equation}  \label{vh.2b}
  V(h) =  h w_1 F + \sum_{j \ge 1} h^{2j+1} \sum_{k=1}^{c(2j+1)} \mu_{2j+1,k} E_{2j+1,k} + i \sum_{j \ge 2} h^{2j} \sum_{k=1}^{c(2j)} \sigma_{2j,k} E_{2j,k}. 
\end{equation}
Here $w_1, \mu_{2j+1,k}, \sigma_{2j,k}$ are \emph{real} polynomials depending on the real and imaginary parts of the parameters $\alpha_i$. 
In other words, the terms of even powers in $h$ in $V(h)$ are pure imaginary, whereas terms of odd powers in $h$ are real.
\end{lemma} 
\begin{proof} 
  We start by noticing that, since $\psis_{h}^{[2]}$ is a time-symmetric integrator, the adjoint $(\psi_{h}^{[r]})^*$
   is precisely the \emph{complex conjugate} of $\psi_{h}^{[r]}$, i.e., $(\psi_h^{[r]})^* = \overline{\psi}_{h}^{[r]}$. In consequence, the corresponding series of operators
  are also identical,
  \[
    \overline{\Psi}^{[r]}(h)  =  (\Psi^{[r]})^*(h)  
  \]
  and therefore $\overline{V}(h) = -V(-h)$. From (\ref{vh.2}), these series are respectively 
\[
\begin{aligned}  
     \overline{V}(h)  & = h \overline{w}_1 F + \sum_{j \ge 1} h^{2j+1} \sum_{k \ge 1} \overline{w}_{2j+1,k} E_{2j+1,k}  + 
       \sum_{j \ge 1} h^{2j} \sum_{k \ge 1} \overline{w}_{2j,k} E_{2j,k}  \\
   -V(-h)  & = h w_1 F + \sum_{j \ge 1} h^{2j+1} \sum_{k \ge 1} w_{2j+1,k} E_{2j+1,k}  -
       \sum_{j \ge 1} h^{2j} \sum_{k \ge 1} w_{2j,k} E_{2j,k},  
 \end{aligned}
 \]
 so that
 \[
         \overline{w}_1 = w_1, \qquad    \overline{w}_{2j+1,k} = w_{2j+1,k}, \qquad  \overline{w}_{2j,k} = - w_{2j,k},
\]  
and (\ref{vh.2b}) is obtained with $\mu_{2j+1,k} = w_{2j+1,k} \in \mathbb{R}$, $\sigma_{2j,k} = \Im(w_{2j,k}) \in \mathbb{R}$.
\end{proof}

\

\noindent
From this lemma one has the following general result concerning the preservation of properties of symmetric-conjugate compositions.

\begin{proposition} \label{prop.2}
Given $\psis_{h}^{[2]}$ a time-symmetric 2nd-order method, let us consider the sym\-me\-tric-conjugate composition
\[
   \psi_{h}^{[r]} =  \mathcal{S}_{\bar{\alpha}_1 h}^{[2]} \circ \mathcal{S}_{\bar{\alpha}_{2} h}^{[2]} \circ \cdots \circ \mathcal{S}_{\alpha_2 h}^{[2]} \circ \mathcal{S}_{\alpha_1 h}^{[2]}
\]
of order $r \ge 3$ and its real part, i.e.,
\begin{equation}  \label{compo.3}
   \hat{R}_{h}^{[2n]} = \frac{1}{2} \left( \psi_{h}^{[r]} + \overline{\psi}_{h}^{[r]} \right),
\end{equation}
applied to the differential equation (\ref{ivp}) with a real vector field $f$.
Then the following statements concerning the pseudo-symmetry and pseudo-sym\-plec\-ti\-ci\-ty of  $\hat{R}_{h}^{[2n]}$ hold:
\begin{itemize}
  \item[(a)] If $\psi_h^{[r]}$ is of odd order, $r=2n-1$, $n=2,3,\ldots$, then  $\hat{R}_{h}^{[2n]}$ 
  is a method of order $2n$ and pseudo-symmetric of order $q = 4n-1$. 
    If in addition $f$ is a (real) Hamiltonian vector field and $\psis_{h}^{[2]}$ is a symplectic integrator, then $\hat{R}_{h}^{[2n]}$ is pseudo-symplectic of order $p=4n-1$.
  \item[(b)] If $\psi_h^{[r]}$ is of even order, $r=2n$, $n=2,3,\ldots$, then  $\hat{R}_{h}^{[2n]}$ is a method of order $2n$ 
  and pseudo-symmetric of order $q=4n+3$. 
  If in addition $f$ is a (real) Hamiltonian vector field and $\psis_{h}^{[2]}$ is a symplectic integrator, then $\hat{R}_{h}^{[2n]}$ is pseudo-symplectic of order $p=4n+3$.
\end{itemize}
\end{proposition}
\begin{proof}
We apply the same strategy as in the proof of Proposition \ref{prop.3}.

\noindent
(a) Since $r=2n-1$, then the series of operators associated with $\psi_{h}^{[r]}$ is 
  $\Psi^{[r]}(h) = \exp(V(h))$, with
\begin{equation}  \label{expre.V}
   V(h) = h F + i h^{2n} V_{2n} + h^{2n+1} V_{2n+1} + i h^{2n+2} V_{2n+2}   + \cdots
\end{equation}
where $V_{2n+j}$, $j=0,1,2,\ldots$ are, according to Lemma \ref{lem.1}, \emph{real} operators in the free Lie algebra
 generated by $\{ F, Y_{3}, Y_{5}, \ldots \}$. From here the series corresponding to $\hat{R}_{h}^{[2n]}$,
\[
    \mathcal{R}^{[2n]}(h) = \frac{1}{2} \left( \exp((V(h)) + \exp(\overline{V}(h)) \right),
\]
can be written as 
\[
    \mathcal{R}^{[2n]}(h) =  \exp \left( \frac{h}{2} F \right) \left(  \frac{1}{2}  \exp((W(h))   + \frac{1}{2} \exp(\overline{W}(h))
      \right) \exp \left( \frac{h}{2} F \right), 
\]
where $W(h)$ is obtained from $\exp(W(h)) = \exp(-h F/2) \, \exp(V(h)) \, \exp(-h F/2)$ as
\[
    W(h) =  i h^{2n} W_{2n} + h^{2n+1} W_{2n+1} + i h^{2n+2} W_{2n+2} + h^{2n+3} W_{2n+3} + i h^{2n+4} W_{2n+4} + \mathcal{O}(h^{2n+5})
\]
with    
\[
\begin{aligned}
     & W_{2n} = V_{2n}, \qquad\quad  W_{2n+1} =  V_{2n+1}, \qquad\quad W_{2n+2} =  V_{2n+2} + \frac{1}{24} [F,[F,V_{2n}]]  \\
     & W_{2n+3} = V_{2n+3} + \frac{1}{24} [F,[F,V_{2n+1}]], \\
     & W_{2n+4} = V_{2n+4} + \frac{1}{24} [F,[F,V_{2n+2}]] + \frac{1}{1920} [F,[F,[F,[F,V_{2n}]]]].
\end{aligned}
\]
In general, terms in $W(h)$ of odd powers in $h$ are real and terms of even powers of $h$ are pure imaginary.
Then, it is clear that
\[
  W(h) + \overline{W}(h) = 2 h^{2n+1} V_{2n+1} + 2 h^{2n+3} \left( V_{2n+3} + \frac{1}{24} [F,[F,V_{2n+1}]] \right) +  \mathcal{O}(h^{2n+5})
\]
and only contains odd powers of $h$. Furthermore,
\begin{eqnarray*}
     \big( W(h) + \overline{W}(h) \big)^2 & = & 4 h^{4n+2} V_{2n+1}^2 + 4 h^{4n+4} \Big( V_{2n+1} ( V_{2n+3} + \frac{1}{24} [F,[F,V_{2n+1}]]) + \\
   & &  +   ( V_{2n+3} + \frac{1}{24} [F,[F,V_{2n+1}]]) V_{2n+1} \Big) + \mathcal{O}(h^{4n+6})  
\end{eqnarray*}
and
\[
  W(h)^2  +  \overline{W}(h)^2   = - 2 h^{4n} V_{2n}^2 + \mathcal{O}(h^{4n+2}).
\]
Proceeding as before, 
\[
   \frac{1}{2} \left( \e^{W(h)}   +  \e^{\overline{W}(h)} \right) - \e^{ \frac{1}{2} \big( W(h) + \overline{W}(h) \big)} = 
    - \frac{1}{2} h^{4n} V_{2n}^2  +  \mathcal{O}(h^{4n+2}),  
\]
so that 
\[
   \mathcal{R}^{[2n]}(h) = \exp \left( \frac{h}{2} F \right)  \exp \left( \frac{1}{2} \big( W(h) + \overline{W}(h) \big) \right) \exp \left( \frac{h}{2} F \right)
   +  \mathcal{O}(h^{4n}),
\]   
whence the conclusions follow readily.

\noindent
(b) We proceed along the same lines as in the preceding case for even order, $r=2n$. Now
\[
   V(h) = h F +  h^{2n+1} V_{2n+1} + i h^{2n+2} V_{2n+2} +  h^{2n+3} V_{2n+3}   + \cdots
\]
and   
\[
    W(h) =   h^{2n+1} W_{2n+1} + i h^{2n+2} W_{2n+2} + h^{2n+3} W_{2n+3} + i h^{2n+4} W_{2n+4} + h^{2n+5} W_{2n+5} + \mathcal{O}(h^{2n+6})
\]
with    
\[
\begin{aligned}
     & W_{2n+1} = V_{2n+1}, \qquad\quad  W_{2n+2} =  V_{2n+2}, \qquad\quad W_{2n+3} =  V_{2n+3} + \frac{1}{24} [F,[F,V_{2n+1}]]  \\
     & W_{2n+4} = V_{2n+4} + \frac{1}{24} [F,[F,V_{2n+2}]], \\
     & W_{2n+5} = V_{2n+5} + \frac{1}{24} [F,[F,V_{2n+3}]] + \frac{1}{1920} [F,[F,[F,[F,V_{2n+1}]]]],
\end{aligned}
\]
whence, as before,
\[
  W(h) + \overline{W}(h) = 2 h^{2n+1} W_{2n+1} + 2 h^{2n+3} W_{2n+3} +  2 h^{2n+5} W_{2n+5} + \mathcal{O}(h^{2n+7}).
\]
On the other hand,
\[
\begin{aligned}
  & W(h)^2  =  h^{4n+2} W_{2n+1}^2 + i h^{4n+3} \big( W_{2n+1} W_{2n+2} + W_{2n+2} W_{2n+1} \big)   \\
    &  + h^{4n+4} \big( W_{2n+1} W_{2n+3} + W_{2n+3} W_{2n+1} - W_{2n+2}^2 \big) \\
    & + i h^{2n+5} \big( W_{2n+1} W_{2n+4} + W_{2n+4} W_{2n+1} + W_{2n+2} W_{2n+3} + W_{2n+3} W_{2n+2} \big) 
+ \mathcal{O}(h^{4n+6}),  
\end{aligned}
\]
whereas
\[
\begin{aligned}
  & \overline{W}(h)^2  =  h^{4n+2} W_{2n+1}^2 - i h^{4n+3} \big( W_{2n+1} W_{2n+2} + W_{2n+2} W_{2n+1} \big)   \\
    &  + h^{4n+4} \big( W_{2n+1} W_{2n+3} + W_{2n+3} W_{2n+1} - W_{2n+2}^2 \big) \\
    & - i h^{2n+5} \big( W_{2n+1} W_{2n+4} + W_{2n+4} W_{2n+1} + W_{2n+2} W_{2n+3} + W_{2n+3} W_{2n+2} \big) 
+ \mathcal{O}(h^{4n+6}).  
\end{aligned}
\]
An straightforward calculation shows that
\[
\begin{aligned}
  & \frac{1}{2} \left( \e^{W(h)}   +  \e^{\overline{W}(h)} \right) - \e^{ \frac{1}{2} \big( W(h) + \overline{W}(h) \big)} = 
   \frac{1}{4} \big( W(h)^2 + \overline{W}(h)^2 \big) - \frac{1}{8}  \big( W(h) + \overline{W}(h) \big)^2 + \cdots \\
   & \quad =  - \frac{1}{2} h^{4n+4} W_{2n+2}^2  +  \mathcal{O}(h^{4n+6})
\end{aligned}   
\]
and finally
\begin{equation}  \label{lep}
   \mathcal{R}^{[2n]}(h) = \exp \left( \frac{h}{2} F \right)  \exp \left( \frac{1}{2} \big( W(h) + \overline{W}(h) \big) \right) \exp \left( \frac{h}{2} F \right)
   +  \mathcal{O}(h^{4n+4}).
\end{equation}   
Now $\hat{R}_{h}^{[2n]}$ is of orden $2n$, but the time-symmetry (and symplecticity) holds up to order $4n+3$.
\end{proof}

\

Although apparently a symmetric-conjugate composition requires solving $N^{[r]}$ equations to achieve order $r$, just as general compositions,
this is not the case, however, when one is interested in projecting on the real axis, since the symmetry in the coefficients introduces additional reductions. 
As Lemma \ref{lem.1} and Proposition
\ref{prop.2} show, for a scheme of order $r=2n$, the $c(2n)$ order conditions at order $h^{2n}$ are pure imaginary and so it is not necessary
to solve them. Therefore, the number of conditions is actually $N^{[2n]}-c(2n)$. This number is collected in the last row of Table \ref{tabla1}. 
This saving in the cost comes of course at the price of reducing the 
preservation of time-symmetry (or symplecticity, etc.) from order $4n+3$ to $4n-1$.

We can proceed in the same vein, since the $c(2n-2)$ order conditions at order $h^{2n-2}$ are also pure imaginary. Now, however, the resulting
schemes after projection are only pseudo-symmetric of pseudo-symmetry order $4n-5$. If $4n-5>2n$, or equivalently if  $2n>5$, then we
still have a method of order $r=2n$ obtained from a symmetric-conjugate composition with  $N^{[2n]}-c(2n)-c(2n-2)$ stages if the
corresponding order conditions have solutions.

This can be generalized as follows:
\begin{proposition}
 Let 
\[
   \psi_{h}^{[r]} =  \mathcal{S}_{\bar{\alpha}_1 h}^{[2]} \circ \mathcal{S}_{\bar{\alpha}_{2} h}^{[2]} \circ \cdots \circ \mathcal{S}_{\alpha_2 h}^{[2]} \circ \mathcal{S}_{\alpha_1 h}^{[2]}
\]
be a symmetric-conjugate composition of order $r=2n$ after projection on the real axis. 
If $4n-(4q+1)>2n$ for some $q\geq 0$ (or equivalently if $2n>4q+1$), then the
number of order conditions to be satisfied by $\psi_{h}^{[r]}$ to get a pseudo-symmetric scheme of 
pseudo-symmetry order $4n-(4q+1)$ after projection on the real axis is
\[
  N^{[2n]}-\sum_{j=0}^{q} c(2n-2j).
\]
\end{proposition}

\

The simplest example corresponds to scheme (\ref{s3a}): Part {\it (a)} of Proposition \ref{prop.2} with $r=3$ 
reproduces the result obtained in \cite{casas21cop} and summarized in section \ref{sec.1}: its real part renders a method
 of order 4 and pseudo-symmetric of order 7.  

If we consider instead a composition (\ref{compo.2}) of order $r=4$, then by taking the real part at each step we do not increase the order, but the pseudo-symmetry order is $q=11$ (instead of $7$). 
In view of Table \ref{tabla1}, it is worth remarking that, although the
symmetric-conjugate compositions require more order conditions to be satisfied than palindromic compositions for orders higher than four, the
methods resulting from projecting on the real axis
require less stages: thus, in particular, it is possible to achieve a 6th-order scheme with only 5 stages, whereas schemes based on palindromic
composition require at least 7 stages.

As an additional illustration, let us take the composition 
\begin{equation}  \label{s3}
     \psi_h^{[4]} = \mathcal{S}_{\bar{\alpha}_1 h}^{[2]} \circ  \, \mathcal{S}_{\alpha_2 h}^{[2]} \circ \, \mathcal{S}_{\alpha_1 h}^{[2]},
\end{equation}
with $s=3$. It is of order $r= 2n = 4$ if
\[
        \alpha_1 = \frac{1}{4} \pm i \frac{1}{4} \sqrt{\frac{5}{3}}, \qquad \alpha_2 = \frac{1}{2}.
\]
Taking its real part, $\Re( \psi_h^{[4]})$, results in a method
also of order 4, but pseudo-symmetric and pseudo-symplectic of order 11. Both schemes $\Re( \psi_h^{[3]})$, (eq. (\ref{ss34})), and $\Re( \psi_h^{[4]})$ are  
of order 4, but whereas the former requires two evaluations of $\psis_{h}^{[2]}$ (instead of three), the latter preserves qualitative properties up to a higher order.

\subsection{Example: harmonic oscillator}

At this point it may be illustrative to apply all the previous 4th-order methods obtained by projecting on the real axis to a  simple
example and check how different behaviors with respect to structure preservation manifest in practice. To this end we choose the
one-dimensional harmonic oscillator,
\[
  q' = p, \qquad p' = -q
\]
with Hamiltonian   
\begin{equation}  \label{hahaos}
H(q,p)=T(p)+V(q)=\frac{1}{2}p^{2}+\frac{1}{2}q^{2}. 
\end{equation}
Denoting by $M_{X}(h)$ the exact matrix evolution associated with the Hamiltonians $X=H$, $T$
and $V$, i.e., $x(h) = (q(h),p(h))^{T}=M_{X}(h)(q(0),p(0))^{T}$, then 
\[
M_{H}(h)=\left( 
\begin{array}{rr}
\cos (h) & \sin (h) \\ 
-\sin (h) & \cos (h)
\end{array}
\right) ,\quad M_{T}(h)=\left( 
\begin{array}{ll}
1 & h \\ 
0 & 1
\end{array}
\right) ,\quad M_{V}(h)=\left( 
\begin{array}{rr}
1 & 0 \\ 
-h & 1
\end{array}
\right) . 
\]
As our basic time-symmetric 2nd-order scheme $\psis_{h}^{[2]}$ we take the leapfrog/Strang integrator 
\begin{equation}  \label{leapfrog}
\psis_{h}^{[2]} = M_{T}(h/2)M_{V}(h)M_{T}(h/2)
\end{equation}
and form the 4th-order schemes $\Re(\psi_h^{[3]})$ (eq. (\ref{ss34})), $\Re(\psi_h^{[4]})$ (eq. (\ref{s3})), and $\Re(\psis_h^{[4]})$
(eq. (\ref{13b})). In this case, it is straightforward
to verify the order of the methods (by computing explicitly the difference $\psi_{h} - M_{H}(h) $ for each method $\psi_h$), the 
pseudo-symmetry order (by evaluating $\psi_{h} \circ \psi_{-h} - I$) and the pseudo-symplecticity order (for instance, by computing the determinant
of the corresponding approximation matrix). In all cases the result agrees with Propositions \ref{prop.3} and \ref{prop.2}.

We can also check the relative efficiency of the three schemes by computing the error in the energy along a time interval with different
time steps. Thus, 
Figure \ref{fig:ho} (top panel) shows this relative error in $H$ as a function of the number of evaluations of the basic second order method
$\psis_{h}^{[2]}$ when $q_0 = 2.5, p_0 = 0$ and the final time is $t_f = 650$. We see that the efficiency of schemes $\Re(\psis_h^{[4]})$ 
and $\Re(\psi_h^{[3]})$ is quite similar for relatively small values of $h$.

It is also illustrative to test the behavior of these schemes for very long time intervals. This is done in Figure \ref{fig:ho} (bottom)
for $t \in [0, 10^7]$ and constant step size $h=1/4$ for $\Re(\psi_h^{[4]})$ and $\Re(\psis_h^{[4]})$, and $h=1/6$ for $\Re(\psi_h^{[3]})$, so that
all schemes require the same computational effort. We see that even for large values of time $\Re(\psi_h^{[4]})$ does not exhibit
a secular component in the error in energy (one might need still larger final times), as is the case for compositions (\ref{compo.1}) involving real coefficients (see \cite{aubry98psr}, where this phenomenon is explained). 
In any case the results are consistent with Proposition \ref{prop.2} and in particular with expression (\ref{lep}).

\begin{figure}[h!] 
\begin{center}
\includegraphics[scale=0.95]{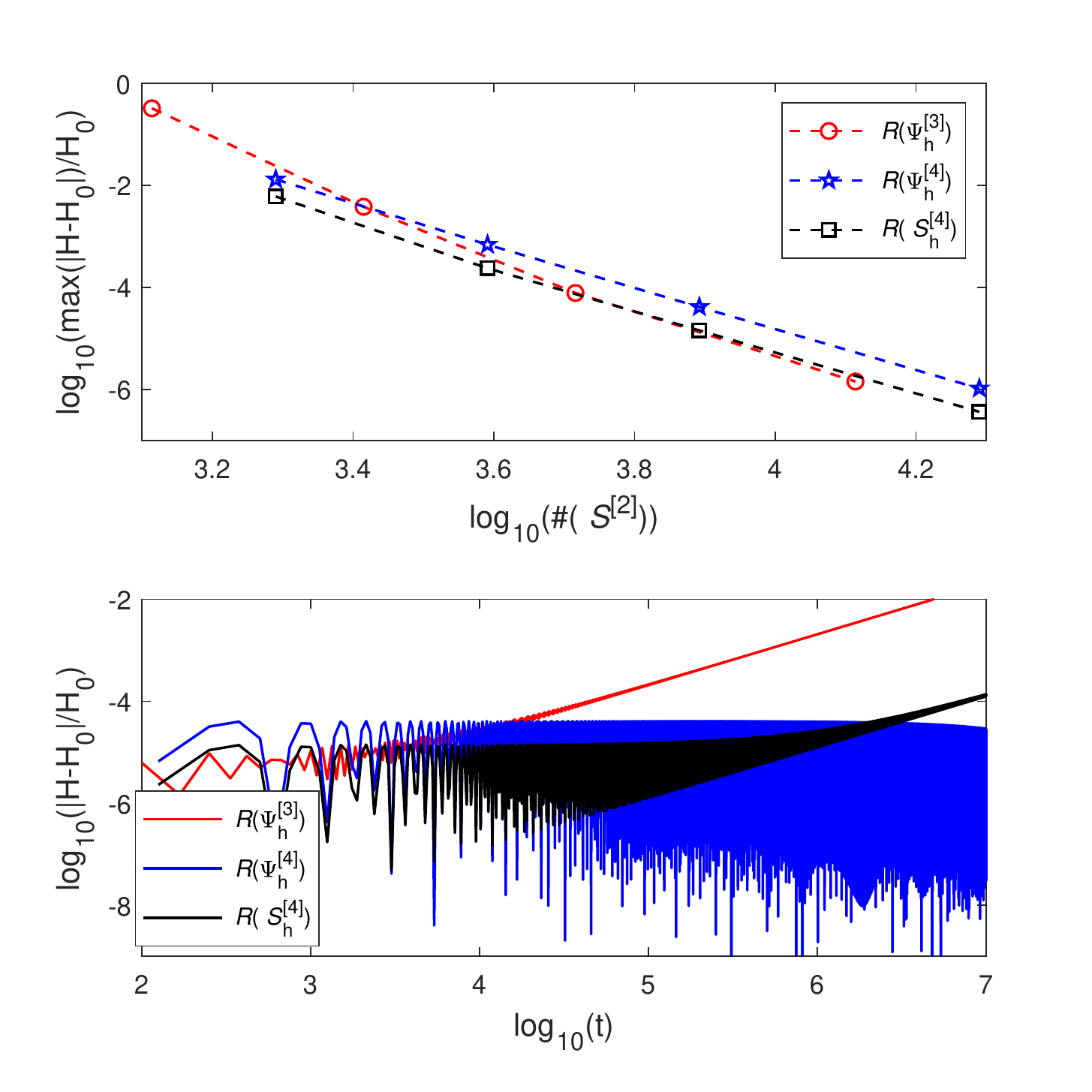}
\end{center}
\caption{\small Top: Relative error in energy vs. the number of evaluations of the basic $\psis_{h}^{[2]}$ scheme for the harmonic
oscillator for $t \in [0, 650]$. Bottom: Evolution of this error along the integration; here the same step size $h=1/4$ is used by 
$\Re(\psi_h^{[4]})$ and $\Re(\psis_h^{[4]})$, and $h=1/6$ by $\Re(\psi_h^{[3]})$. \label{fig:ho}}
\end{figure}

\section{Symmetric-conjugate composition methods obtained from a 2nd-order symmetric basic scheme}
\label{sec.3}

In this section we propose new methods constructed  from a basic time-symmetric 2nd-order basic scheme by  symmetric-conjugate composition. 
Since the case of
order 4 has been already analyzed in section \ref{sec.2}, here we study compositions with $s \ge 4$
stages. From Proposition \ref{prop.2} it is clearly advantageous to consider conjugate-symmetric compositions of odd order $r=2n-1$,
since taking the real part
leads automatically to a method of order $r=2n$ (but requiring only the computational cost of a method of order $2n-1$).

For simplicity, we denote in the sequel the general composition (\ref{compo.1}) by its sequence of coefficients:
\[
   (\alpha_s, \alpha_{s-1}, \ldots, \alpha_2, \alpha_1). 
\]
As a general rule for selecting a particular method, we follow the same criterion as in \cite{blanes13oho}, 
namely we first choose a subset of solutions with
small 1-norm of the coefficient vector $(\alpha_s, \ldots, \alpha_1)$ and, among them, choose the one that minimizes the norm of 
the main term in the corresponding truncation error.

\paragraph{Order 6.}

According to the previous treatment, one could consider in principle a symmetric-conjugate composition verifying the order conditions
\[
  w_1 = 1, \qquad w_{3,1} = 0, \qquad w_{5,1} = w_{5,2} = 0
\]
in (\ref{orcon1}), since $w_{4,1}$ is pure imaginary, so that when taking the real part of the composition, it does not contribute to the error. Four stages
would then be necessary to construct a 6th-order method. It turns out, however, that these equations do not admit solutions with the required
symmetry $\alpha_4 = \bar{\alpha}_1$, $\alpha_3 = \bar{\alpha}_2$, and thus at least  
 $s=5$ stages are necessary. The additional parameter can be used to solve the condition $w_{4,1}=0$ so as to achieve order 5. These
equations admit 
 5 solutions (plus the corresponding complex conjugate) for the sequence (\ref{compo.2}), i.e., for 
\begin{equation}   \label{5os}
 \psi_h^{[5]} =  \mathcal{S}_{\bar{\alpha}_1 h}^{[2]} \circ  \mathcal{S}_{\bar{\alpha}_2 h}^{[2]} \circ \mathcal{S}_{\alpha_3 h}^{[2]} \circ \mathcal{S}_{\alpha_2 h}^{[2]} \circ 
   \mathcal{S}_{\alpha_1 h}^{[2]}. 
\end{equation}
Among them, we select      
\[
\aligned
  &  \alpha_1 = 0.1752684090720741140583563 + 0.05761474413053870201304364 \,  i \\
  &  \alpha_2 = 0.1848736801929841604288898 - 0.1941219227572495885067758 \, i \\
  &  \alpha_3 = 0.2797158214698834510255077
\endaligned
\]
so that the real part
\[
   \hat{R}_h^{[6]} = \frac{1}{2} \left( \psi_{h}^{[5]} + \overline{\psi}_{h}^{[5]} \right)
\]
leads to a method of order 6 which, according with Part \textit{(a)} of Proposition \ref{prop.2}, 
is pseudo-symmetric and pseudo-symplectic of order 11, although it only has 5 stages (one of them being real). Notice that, according to 
Table \ref{tabla1}, $s=7$ stages are required to construct a conjugate-symmetric composition of order 6. Such a method was indeed
proposed and tested on several numerical examples in \cite{blanes10smw}, exhibiting a good long time behavior.  This behavior can
be explained by Proposition \ref{prop.2}, since
the corresponding method $\hat{R}_h^{[6]}$ constructed  by taking its real part is pseudo-symmetric and pseudo-symplectic of order 15.

The same number of stages ($s=7$) is also required by a palindromic composition to solve the 4 order conditions necessary to achieve order 6.
As shown in \cite{blanes13oho}, the best solution within this class is the composition S$_76$ previously found in \cite{chambers03siw}. 
By taking the real part, the
corresponding scheme $\phi_h^{[6]}$ is pseudo-symmetric of order 13 and involves 2 more stages than $\hat{R}_h^{[6]}$.

\paragraph{Order 8.}

In view of the structure of the series of operators $\exp(V(h))$ associated with a symmetric-conjugate composition, eq. (\ref{vh.2b}), it is clear
that if the order conditions
\begin{equation}  \label{oc8.1}
\begin{aligned}
   & w_1 = 1, \qquad w_{3,1}=0, \qquad w_{4,1} = 0, \qquad w_{5,1}=w_{5,2}=0, \\
   &  w_{7,1} = w_{7,2} = w_{7,3} = w_{7,4} = 0
\end{aligned}   
\end{equation}
are satisfied by $\psi_h^{[r]}$, then we get a 5th-order composition whose projection on the real axis is an 8th-order 
approximation. Here 
the condition $w_{4,1} = 0$ has to be included, since otherwise there appears a contribution in $h^8$. In consequence, at least $s=9$ stages
are necessary to solve equations (\ref{oc8.1}). We have in fact found 7 solutions (+ c.c.) with the required symmetry and positive real part. Among
them, we propose, according with the previous criteria, 
\begin{equation}  \label{order8.9}
  \aligned
         &  \alpha_1 =  \bar{\alpha}_{9} =  0.08848457824129988495666830 -  0.07427185309152124718276000 \ i \\
         &  \alpha_2 =  \bar{\alpha}_{8} =  0.15956870501880174198291033 + 0.02322565281009720913454462 \ i \\
         &  \alpha_3 =  \bar{\alpha}_7 =    0.09359461460849451904251162 + 0.13796356924496549819619086 \ i \\
         &  \alpha_4 =  \bar{\alpha}_6 =    0.15769224955121857774144315 -  0.07166960107892295549940996 \ i \\
         &  \alpha_5 =                               0.00131970516037055255293318         
  \endaligned
\end{equation}        
We thus have an 8th-order scheme obtained from a symmetric-conjugate 
composition of a basic 2nd-order time symmetric scheme requiring only 9 stages. This is the reason for the last entry in Table \ref{tabla1}.
Since the composition is of order 5, the final scheme will be
pseudo-symmetric of order 11. In case one is interested in preserving properties up to a higher order, then two more stages are necessary to
solve the order conditions at order 6. In that case, we have a symmetric-conjugate composition of order 7 involving $s=11$ stages which is
pseudo-symmetric of order 15. 

By contrast, $s=15$ stages are required to solve the 8 order conditions of an 8-th order left-right palindromic composition. 
In \cite{blanes13oho}, an optimized method
 of this class is proposed. Notice that, when one takes its real part, the final method is pseudo-symmetric of order 17. In any case, this
different behavior with respect to time-symmetry will be hardly visible in most practical situations. 
 
 We have carried out a numerical search of solutions such an 11-stage symmetric-conjugate composition, finding 29 sets of coefficients with
 positive real part. Among them, we recommend the following:
\begin{equation}  \label{order8}
  \aligned
         &  \alpha_1 =  \bar{\alpha}_{11} =  0.07683292597738736205503 - 0.05965805084613860757735 \ i \\
         &  \alpha_2 =  \bar{\alpha}_{10} =  0.12844482070368650612973 + 0.02479812697572531668668 \ i \\
         &  \alpha_3 =  \bar{\alpha}_9 =  0.06855723904168450389158 + 0.11276129325339482617990 \ i \\
         &  \alpha_4 =  \bar{\alpha}_8 =  0.11879414810128891257046 - 0.04055765731534572031090 \ i \\
         &  \alpha_5 =  \bar{\alpha}_7 =  0.10279469076169306832515 + 0.06735917341353737963638 \ i \\
         &  \alpha_6 =    0.009152350828519294056116       
  \endaligned
\end{equation}        

A method of order 10 within this family would require at least 17 stages, since one has to construct a symmetric-conjugate composition of order
5 (5 order conditions) also verifying the 4 conditions at order 7 and the 8 conditions corresponding to order 9. This method would be pseudo-symmetric
of order 11. The pseudo-symmetry can be raised up to order 15 by adding the 2 conditions at order 6 for a total of 19 stages. By contrast, a 
palindromic composition requires a minimum of 31 stages.

\section{Numerical examples}

\subsection{Kepler problem}
As a first example we take  the two-dimensional
Kepler problem with Hamiltonian
\begin{equation*}
   H(q,p) = T(p) + V(q) = 
	\frac{1}{2} p^T p - \mu \frac{1}{r}.
\end{equation*}
Here $q=(q_1,q_2), p=(p_1,p_2)$, $\mu=GM$, $G$ is the gravitational constant and $M$ is the sum of the masses of the two bodies. 
We take $\mu=1$ and initial conditions
\begin{equation*}
  q_1(0) = 1- e, \quad q_2(0) = 0, \quad p_1(0) = 0, \quad p_2(0) = \sqrt{\frac{1+e}{1-e}},
\end{equation*}
so that the trajectory corresponds to an ellipse of eccentricity $e=0.6$, and integrate with the 6th- and 8th-order methods resulting from
symmetric-conjugate and palindromic compositions after projecting on the real axis at each step. We denote them by $\mathcal{S}_p^{r(*)}$
and $\mathcal{S}_p^{r}$, respectively, where $r$ is the order of the method and $p$ is the number of stages (basic 2n-order integrators) involved
in the composition. Thus, 
\begin{itemize}
  \item $\mathcal{S}_5^{6(*)}$ refers to scheme (\ref{5os}); 
  \item $\mathcal{S}_7^{6(*)}$ is method $\mathrm{S}_7^{*6}$ of \cite{blanes10smw};
  \item $\mathcal{S}_7^{6}$ corresponds to composition $\mathrm{S}_76$ found in \cite{chambers03siw};
  \item $\mathcal{S}_{9}^{8(*)}$ refers to method (\ref{order8.9});
  \item $\mathcal{S}_{11}^{8(*)}$ denotes method (\ref{order8}); 
  \item $\mathcal{S}_{15}^{8}$  corresponds to composition $\mathrm{S}_{15}8$ obtained in \cite{blanes13oho}.
\end{itemize}
In our fist experiment we fix the final time $t_f = 650$ and compute the maximum of the relative error in the energy along the trajectory for
different step sizes. Thus, we end up with Figure \ref{fig:cost2} (top), which shows this relative error in energy vs. the
number of basic 2nd-order methods necessary for each scheme. 

\begin{figure}[h!] 
\begin{center}
\includegraphics[scale=0.92]{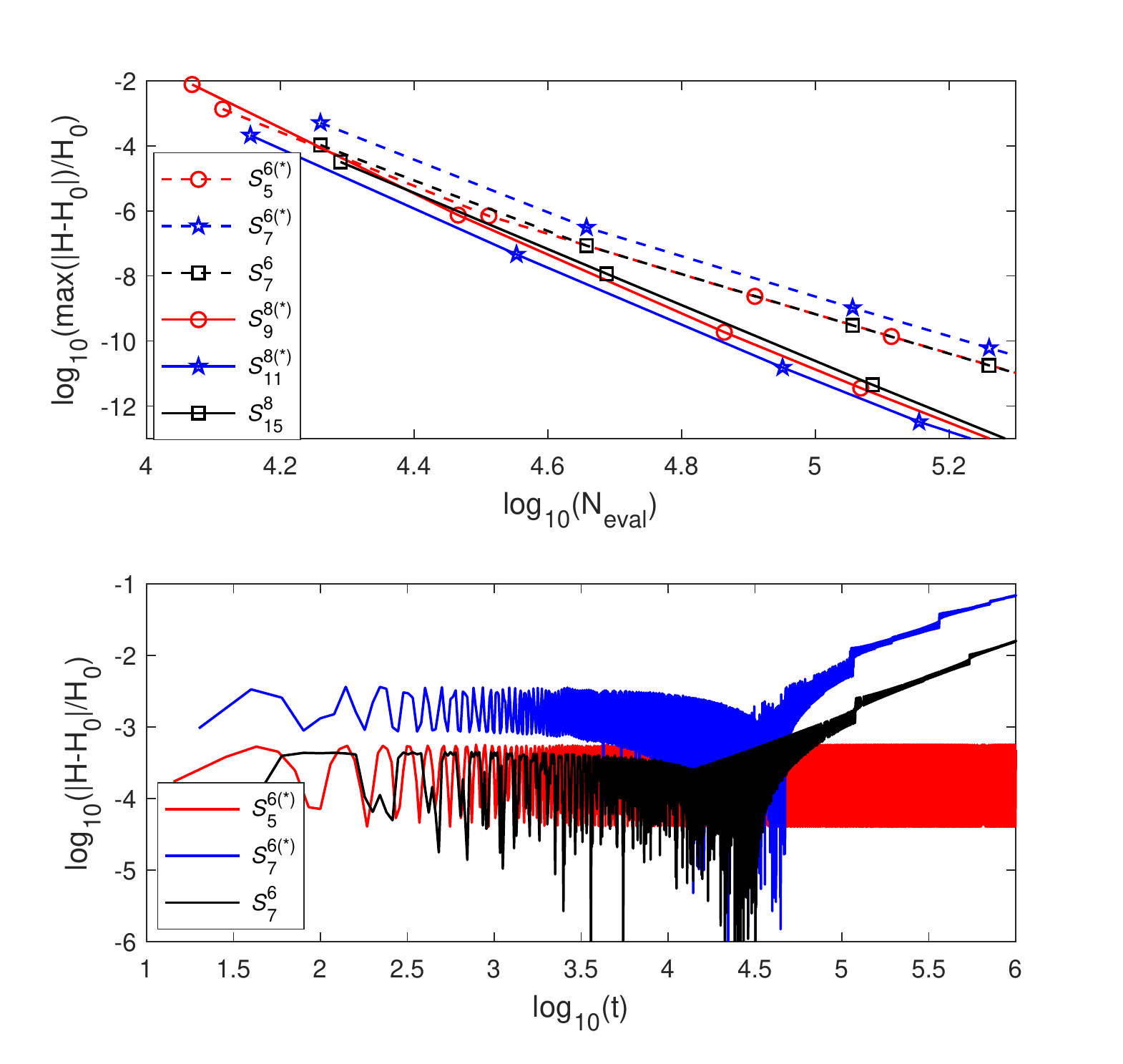}
\end{center}
\caption{\small Top: Relative error in energy vs. the number of evaluations of the basic $\psis_{h}^{[2]}$ scheme for the Kepler problem. Bottom: Evolution of this error along the integration of 6th-order methods. 
\label{fig:cost2}}
\end{figure}

Notice that the new 8th-order schemes obtained from symmetric-conjugate compositions are almost one order of magnitude
 more efficient than $\mathcal{S}_{15}{8}$
coming from a palindromic composition, due to the reduced number of basic 2nd-order integrators they require. In addition, it is also worth
remarking that these 8th-order methods work better than 6th-order methods even for large time steps, in contrast with what usually happens
with compositions with real coefficients.

In Figure \ref{fig:cost2} (bottom) we illustrate the long-time behavior of the previous 6th-order schemes. To this end, for the same initial
conditions, we integrate until the final time $t_f = 10^6$ with a constant step size in such a way that all methods involve the same
number of evaluations of the basic integrator. Specifically, $h= 2/5$ for both $\mathcal{S}_7^{6(*)}$ and $\mathcal{S}_7^{6}$, whereas
$h= 2/7$ for $\mathcal{S}_5^{6(*)}$. We see that the latter behaves as a symplectic integrator for the whole integration interval.

\subsection{The pendulum}

We consider next the one-dimensional pendulum with Hamiltonian
\begin{equation*}
   H(q,p) = T(p) + V(q) = 
	\frac{1}{2} p^2 + (1-\cos(q)).
\end{equation*}
We take as initial conditions $q_0=0, p_0=\alpha$, such that for small values of $\alpha$ this is close to a harmonic oscillator,
 whereas for $\alpha>2$ the pendulum gives full turns. We take $\alpha=\frac12$ (small oscillations) and $\alpha=5$ (full turns), integrate until $t_f=200\pi$ and measure the average error in energy as well as the average two-norm error in $q,p$ at times $t=k\cdot 2\pi, \ k=1,2,\ldots,100$ versus the number of stages. The results are shown in Figure~\ref{fig:pendulum}. We also observe the superiority of the higher order methods for nearly all accuracies and, among the eighth-order schemes, $\mathcal{S}_{11}^{8(*)}$ shows the best performance in all cases we have considered.

\begin{figure}[h!]
  \begin{center}
       \includegraphics[scale=0.74]{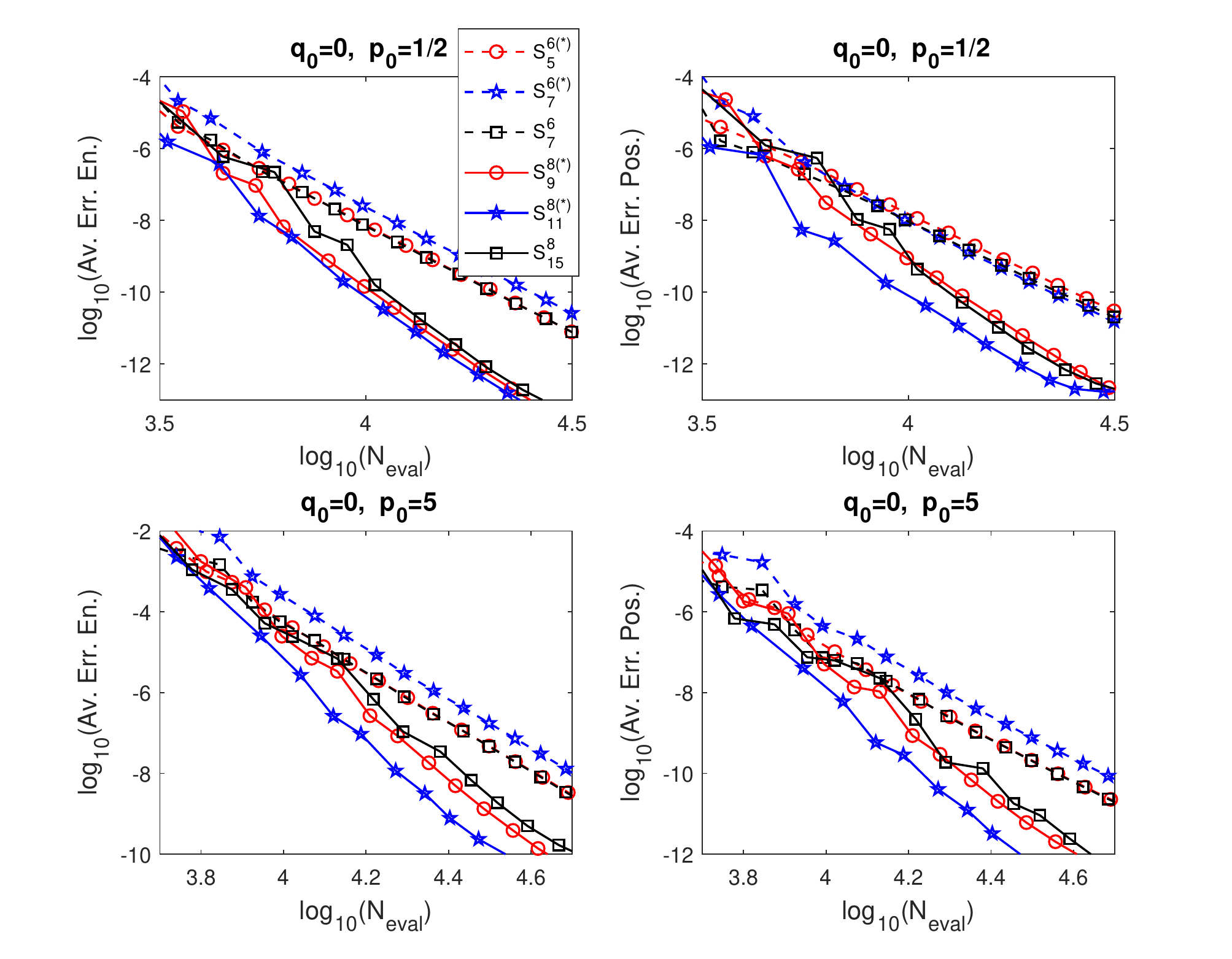}
      \caption{Average relative error in energy (left figures) and average error in positions (right) vs. the number of evaluations of the basic $\psis_{h}^{[2]}$ scheme for the pendulum.}
    \label{fig:pendulum}
  \end{center}
\end{figure}

\section{Stability}

Efficiency diagrams of Figures \ref{fig:cost2} and \ref{fig:pendulum} show a distinctive pattern: methods of order 8 are more efficient than schemes
of order 6 not only for small values of $h$, but in fact for the whole region of $h$ where errors are of practical interest. This comes in contrast with
what happens for methods with real coefficients: in that case the error (in a log-log plot) of a given integrator typically exhibits a corner where
higher error terms contribute by the same amount as the main error term. In this way, the errors of the different schemes form an envelope and one
is interested in selecting those particular methods lying close to this envelope.  

In reference \cite{mclachlan02foh} McLachlan presents a simple
model to determine in first approximation this corner by defining the \emph{elbow} of a given method as a crude estimate for the envelope and for
the nonlinear stability of the method. The idea is as follows: if one assumes that all vector fields $Y_j$ in (\ref{gc.1}) have the same order of
magnitude, and considers only a single error term $e_j$ at each order for a given palindromic composition (\ref{compo.2b}) of order $r$, 
then this effective error scales as 
\[
\mathcal{E} :=h^r e_{r+1} + h^{r+2} e_{r+3} + \cdots.
\]
Here $e_{j}$ includes a factor $s^r$ multiplying the error coefficient of
the $s$-stage composition,   so that it can be compared
to the reference value $1$ for the basic method $\psis_h^{[2]}$. Then the elbow is defined as
\[
   h^* := \sqrt{ \frac{e_{r+1}}{e_{r+3}}}
\]
thus indicating the value of $h$ below which the asymptotic error $\mathcal{O}(h^r)$ is observed, so that no method should be used
with time steps larger than $h^*$. What is remarkable about this model is that
both $\mathcal{E}$ and $h^*$ provide a good qualitative picture of palindromic compositions of different orders \cite{mclachlan02foh}.

We have carried out a similar treatment for the compositions (both palindromic and symmetric-conjugate) with complex coefficients of this work
and the corresponding results are collected in Table \ref{table2}. Symmetric-conjugate compositions are denoted by SC, whereas PR and PC
stand for palindromic compositions with real and complex coefficients, respectively. We also collect in the last column the effective stability
limit, i.e the supremum of the step sizes $h$ for which the approximate solution matrix for the harmonic oscillator
furnished by each scheme may be bounded independently of 
the iteration $n$ so that the error does not
grow exponentially as $n$ increases. The reference values of $h^*$ and $h_t/s$ for the basic integrator $\psis_h^{[2]}$
are respectively 1 and 2.

\begin{table}[ht]
\begin{center}
\begin{tabular}{|c|c|c|c|c|c|}\hline
 \multicolumn{6}{|c|}{\bfseries Order 4}\\ \hline
 Method & $s$ & $e_5$ & $e_7$ & $h^*$ & $h_t/s$  \\ \hline
  SC  & 2  & 1.7778 & 2.3704 & 0.8660 & 1.7320 \\
  SC & 3 & 2.2500 & 8.4375 & 0.5164 & 0.8622 \\
  PR & 3 & 428.60 & 18222 & 0.1534 & 0.5245 \\
  PC & 3 & 1.9562 & 3.0189 & 0.8050 & 1.3771 \\ \hline
\end{tabular}  

\

\begin{tabular}{|c|c|c|c|c|c|}\hline
 \multicolumn{6}{|c|}{\bfseries Order 6}\\ \hline
 Method & $s$ & $e_7$ & $e_9$ & $h^*$ & $h_t/s$  \\ \hline
  SC  & 5  & 4.4951 & 44.651 & 0.3173 & 0.6172 \\
  SC & 7 & 4.5667 & 147.577 & 0.1759 & 0.4457 \\
  PR & 7 & 104518 & $9.7 \times  10^6$ & 0.1038 & 0.3242 \\
  PC & 7 & 4.3876 & 92.115 & 0.2182 & 0.4482 \\ \hline
\end{tabular}  

\

\begin{tabular}{|c|c|c|c|c|c|}\hline
 \multicolumn{6}{|c|}{\bfseries Order 8}\\ \hline
 Method & $s$ & $e_9$ & $e_{11}$ & $h^*$ & $h_t/s$  \\ \hline
  SC  & 9  & 14.060 & 5.996 & 1.5312 & 0.8638 \\
  SC & 11 & 7.4082 & 2.4572 & 1.7363 & 0.9353 \\
  PC & 15 & 2.0506 & 10.429 & 0.4434 & 0.7896 \\ \hline
\end{tabular}  
\end{center}
  \caption{\small{Scaled error coefficients for different compositions of order 4, 6 and 8 with complex and real coefficients. $s$ is the number of stages,
  $h^*$ is the elbow of
  the method and $h_t/s$ corresponds to the effective linear stability limit. SC refers to symmetric-conjugate compositions, whereas PR and PC stand
  for palindromic compositions with real and complex coefficients, respectively.}}
\label{table2}
\end{table}  
We also depict in Figure \ref{fig:elbow} the effective error $\mathcal{E}$ vs. $1/h$ for the basic scheme $\psis_h^{[2]}$ and several compositions with
complex coefficients of order 4 (dash-dotted lines), 6 (dashed) and 8 (solid lines) 
whose errors terms are collected in Table \ref{table2}. For comparison we also include the curve
corresponding to the triple-jump of order 4 with real coefficients (dotted line).

In view of Table \ref{table2} and Figure \ref{fig:elbow} some comments are in order. 
First, the size of the scaled error terms are much smaller for compositions with
complex coefficients than for schemes with real coefficients. Second, these error terms grow only moderately with the order
for a given method, in contrast with compositions involving real coefficients. In some cases (e.g., for symmetric-conjugate compositions of
order 8) they even decrease in size. Third, as a result, the elbow $h^*$ is typically much larger for schemes with complex coefficients,
attaining values for which the error is quite considerable. As a consequence, the asymptotic behavior of the error for this class of methods is
already visible for all practical values of the step size in a given integration. This can be clearly seen in Figure \ref{fig:elbow}, which 
qualitatively reproduces quite well the behavior observed for the Kepler and pendulum problems (Figures \ref{fig:cost2}, \ref{fig:pendulum}): we
notice that the curves corresponding to the 8th-order symmetric-conjugate compositions are placed below the one given by the basic scheme
$\psis_h^{[2]}$ for all relevant errors. 

\begin{figure}[t]
  \begin{center}
       \includegraphics[scale=0.65]{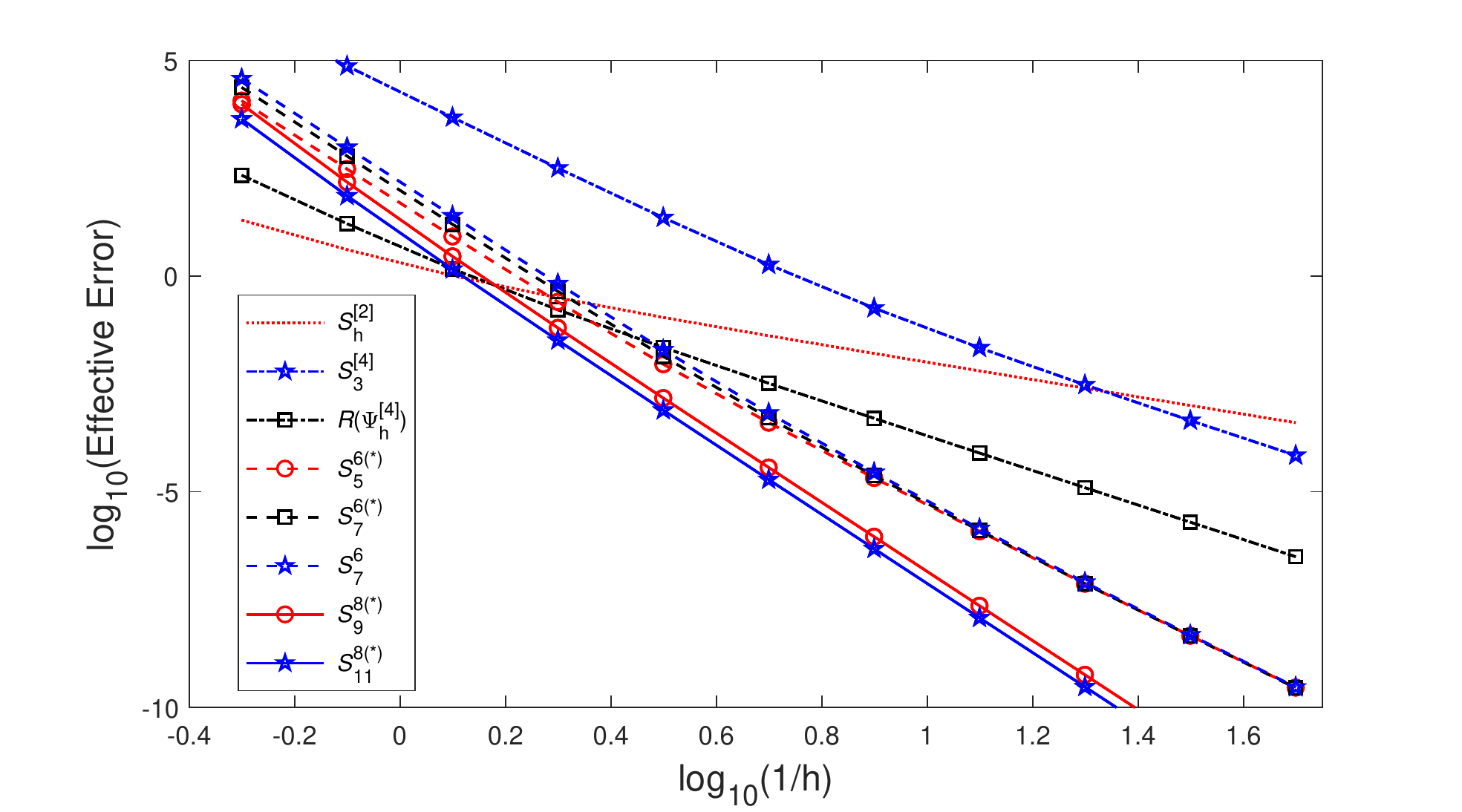}
      \caption{Nominal effective error $\mathcal{E}$ vs. $1/h$ for different compositions with complex coefficients of order 4, 6 
      and 8. Triple-jump of order 4 with real coefficients (dash-dotted line with stars) is included for comparison. The order of the methods is clearly visible.}
    \label{fig:elbow}
  \end{center}
\end{figure}

\section{Concluding remarks}

Although compositions of basic second-order time-symmetric integrators $\psis_h^{[2]}$ involving complex coefficients have been proposed in the past 
for overcoming the difficulties associated with the presence of negative real coefficients when the order $r \ge 3$, this is, we believe, the first systematic 
analysis of such composition methods. 

When the vector field defining the differential equation is real, the goal is of course to get accurate \emph{real} approximations to the exact solution,
whereas the direct application of a composition method with complex coefficients leads in general to a \emph{complex} approximation at each step.
Two approaches present themselves in a natural way: either one projects the solution on the real axis at the end of each integration step or
 the numerical solution is only projected at the end of the integration interval (or more generally only when output is required). In either case, however,
the favorable preservation properties the composition inherits from the basic scheme (such as time-symmetry, symplecticity, volume preservation, etc.)
are generally lost and the question is characterizing this loss in a precise way. 

We have seen that, in general, projecting at each time step preserves these qualitative properties up to an order much higher than the order of
accuracy of the composition itself, and provides a good description of the system. In addition to the usual palindromic sequence of 
coefficients in a composition, we have also explored symmetric-conjugate sequences, showing that it is indeed possible to construct numerical
integrators of high order requiring a smaller number of basic schemes. Thus, in particular, we have present a 6th-order method requiring 5
$\psis_h^{[2]}$ evaluations, and an 8th-order scheme involving only 9 $\psis_h^{[2]}$ evaluations. These numbers have to be compared with 7 and 15,
respectively, for palindromic compositions. The numerical tests carried out clearly illustrate how this reduction in the computational complexity 
translates into a better performance whereas still sharing with the exact solution its main qualitative properties up to a higher order. Moreover,
the efficiency diagrams show that higher order methods involving complex coefficients are more efficient than lower order schemes, not only for
small values of the step size $h$ as occurs typically with real coefficients, but in the whole region of $h$ where errors are reasonably small.
This remarkable property has been traced back to the structure and size of the successive terms in the asymptotic expansion of the error of these
compositions.

Since high order methods obtained from compositions with complex coefficients provide good accuracy and behave in practice as 
geometric numerical integrators, one might consider comparing them with composition methods with real coefficients on practical applications.
Take, for instance, the 8th-order method $\psis_{9}^{8(*)}$, involving 9 basic schemes $\psis_h^{[2]}$. The minimum number for a composition
method of the same order with real coefficients is 15, and more are required to have efficient schemes. It might be the case that for certain problems
this reduction in the number of evaluations compensates the extra cost due to using complex arithmetic, although this of course is highly
dependent of the particular structure of the processor and the implementation. In any case, this will be the subject of future research.

When dealing with this class of schemes, one might contemplate the possibility of projecting at the end of the whole integration interval 
or alternatively after
$N$ time steps, with $t = N h$, instead of projecting after each step. In that case, however, the approximate numerical solution explores along the evolution
regions
in the complex plane not necessarily in the proximity of the real axis, so that a rigorous analysis is more involved. Preliminary
results show that even in such a situation one might still have preservation of structures depending on the particular system,
the step size and the initial conditions one is considering. This issue deserves further analysis and will be explored in a forthcoming paper.

\subsection*{Acknowledgements}
FC and SB would like to thank the Isaac Newton Institute for Mathematical Sciences for support and hospitality during the programme ``Geometry,
compatibility and structure preservation in computational differential equations'', when work
on this paper was undertaken. This work was supported by EPSRC Grant Number EP/R014604/1 and by 
Ministerio de Ciencia e Innovaci\'on (Spain) through project PID2019-104927GB-C21 (AEI/FEDER, UE). 
A.E.-T. has been additionally supported by the predoctoral contract BES-2017-079697 (Spain).

\bibliographystyle{siam}

\end{document}